\documentclass[12pt]{article}
\usepackage{amsthm, amssymb, amstext}
\usepackage[fleqn]{amsmath}
\usepackage{latexsym}
\usepackage{bm}
\usepackage{graphicx}
\usepackage{hyperref}
\usepackage{mathtools}
\usepackage{enumerate}
\usepackage[dvipsnames]{xcolor}
\usepackage{xifthen}

\newtheorem{theorem}{Theorem}
\newtheorem{lemma}[theorem]{Lemma}
\newtheorem{observation}[theorem]{Observation}
\newtheorem{corollary}[theorem]{Corollary}
\newtheorem{conjecture}[theorem]{Conjecture}
\newtheorem{problem}[theorem]{Problem}
\newtheorem{remark}[theorem]{Remark}

\theoremstyle{definition}

\newcommand{\rby}[2][]{\xrightarrow{\text{L\ref{lem:#2}}\ifthenelse{\isempty{#1}}{}{,#1}}}

\title{An update on reconfiguring $10$-colorings of planar graphs}
\author{Zden\v{e}k Dvo\v{r}\'ak\thanks{Computer Science Institute of Charles University, Prague, Czech Republic, email: \texttt{rakdver@iuuk.mff.cuni.cz} }\and
Carl Feghali\thanks{Computer Science Institute of Charles University, Prague, Czech Republic, email: \texttt{feghali.carl@gmail.com} }}

\begin{document}   

\maketitle

\begin{abstract}
The reconfiguration graph $R_k(G)$ for the $k$-colorings
of a graph~$G$ has as vertex set the set of all possible proper $k$-colorings
of $G$ and two colorings are adjacent if they differ in the color of exactly one vertex.
A result of Bousquet and Perarnau (2016) regarding graphs of bounded degeneracy implies
that if $G$ is a planar graph with $n$ vertices, then $R_{12}(G)$ has diameter at most $6n$.
We improve on the number of colors, showing that $R_{10}(G)$ has diameter at most $8n$ for every planar graph $G$ with $n$ vertices.
\end{abstract}

\section{Introduction and result} 

Let $G$ be a graph, and let $k$ be a non-negative integer. 
A (proper) \emph{$k$-coloring} of $G$ is a function $\varphi: V(G) \rightarrow \{1, \dots, k\}$ such that $\varphi(u) \neq \varphi(v)$ whenever $uv \in E(G)$. 
The reconfiguration graph $R_k(G)$ of the $k$-colorings of $G$ has as vertex set the set of all $k$-colorings of $G$,
with two colorings adjacent if they differ in the color of exactly one vertex.  That is, two $k$-colorings $\varphi_1$ and $\varphi_2$ are joined by a path
in $R_k(G)$ if and only if we can transform $\varphi_1$ into $\varphi_2$ by recoloring vertices one by one, always keeping the coloring proper,
and the number of recolorings needed is equal to the distance between $\varphi_1$ and $\varphi_2$ in $R_k(G)$.  Hence, it is natural to ask how
the diameter of $R_k(G)$ depends on $k$ and the number of vertices of $G$, subject to various conditions ensuring the $k$-colorability of $G$.

The study of the reconfiguration graph for colorings was begun by the
statistical physics community  in the context of Glauber dynamics for random
colorings; see for example~\cite{jerrum, vigoda}.  It has also recently
attracted the attention because of its connections to the
existence of FPTAS for the number of colorings, but also for its own
structural and computational merit.  For example, typical questions include
deciding whether two colorings belong to the same component of the
reconfiguration graph, or that of determining the diameter of its components.
For more details, we refer the reader to the surveys by van den Heuvel~\cite{He13} and by Nishimura~\cite{nishimura}.

A graph is \emph{$k$-degenerate} if every subgraph of the graph contains a vertex of degree at most $k$. 
Clearly, every $k$-degenerate graph $G$ is $(k+1)$-colorable, but $R_{k + 1}(G)$ may be disconnected
(e.g. in the case $G = K_{k + 1}$, but there are many more instances~\cite{bonamy12}).
On the other hand, $R_{k + 2}(G)$ is always connected~\cite{dyer}.  Cereceda~\cite{luisthesis} conjectured the following. 

\begin{conjecture}\label{conj}
If $G$ is a $k$-degenerate graph on $n$ vertices, then $R_{k + 2}(G)$ has diameter $O(n^2)$. 
\end{conjecture}

This bound would be best possible~\cite{BJLPP14}.  Although the conjecture has resisted several efforts, there have been some partial results
surrounding it~\cite{bonamy13, bousquet11, heinrich, eiben, feghalipaths, feghali, brooks}.
The most important breakthrough is a theorem of Bousquet and Heinrich~\cite{heinrich} where it was shown,
amongst other results, that $R_{k + 2}(G)$ has diameter $O(n^{k + 1})$. In particular, the conjecture is still open even for $k = 2$.  

Bousquet and Perarnau~\cite{bousquet11} gave the following bound in the situation when the number of colors is substantially larger than $k + 2$.
\begin{theorem}[{Bousquet and Perarnau~\cite[Theorem~1]{bousquet11}}]\label{thm:bousquet}
If $G$ is a $k$-degenerate graph on $n$ vertices and $c\ge 2k+2$, then $R_c(G)$ has diameter at most $(k+1)n$. 
\end{theorem}
It was also shown by Bartier and Bousquet~\cite{bartier} that $R_{k + 4}(G)$ has diameter $O(n)$
for every $k$-degenerate chordal graph $G$ of bounded maximum degree. Another result in this direction was
obtained by the second author~\cite{fsparse} by showing that $R_{k + 2}(G)$ has diameter $O(n (\log n)^{k + 1})$ for every graph $G$
of maximum average degree at most $k + \epsilon$ ($0 \leq \epsilon < 1$).

Planar graphs are $5$-degenerate and have maximum average degree less than $6$, and thus the aforementioned
results imply that if $G$ is a planar graph with $n$ vertices, then $R_8(G)$ has diameter $O(n (\log n)^{7})$
and $R_{12}(G)$ has diameter at most $6n$.
This motivates the following question.
\begin{problem}\label{pbm:main}
What is the minimum integer $\kappa$ such that for every planar graph $G$ with $n$ vertices,
$R_\kappa(G)$ has diameter $O(n)$?
\end{problem}
The object of this paper is to show $\kappa\le 10$, improving on the bound $12$ following from Theorem~\ref{thm:bousquet}.

\begin{theorem}\label{thm:main}
Let $G$ be a planar graph on $n$ vertices. Then $R_{10}(G)$ has diameter at most $8n$. 
\end{theorem}

Consider the coloring of the icosahedron graph $D$ where the opposite vertices get the same color.
This gives a 6-coloring of $D$ where the closed neighborhood of each vertex contains all 6 colors,
and hence this 6-coloring forms an isolated vertex in $R_6(D)$.  Consequently, $R_6(G)$ does not
even need to be connected for planar graphs, implying $\kappa\ge 7$.  However, not much is known
about $R_7(G)$ for planar graphs $G$. The 5-degenerate graphs for which $R_7(G)$ has quadradic diameter
constructed in~\cite{BJLPP14} (paths with four apex vertices) are non-planar.  A natural candidate
for a planar graph $G$ with $R_7(G)$ of quadratic diameter is as follows: Consider the drawing of $K_7$ on
the torus.  Cut this drawing along a non-contractible triangle and glue together many copies of the resulting
cylinder.  We obtain a planar graph with a $7$-coloring such that the closed neighborhood of all but six vertices
contains all $7$ colors, so to recolor this graph, one has to ``propagate'' from the ends of the cylinder.
However, this graph $G$ is $3$-degenerate and chordal, and thus $R_7(G)$ in fact has linear diameter by the aforementioned
result of Bartier and Bousquet~\cite{bartier}.  Hence, we cannot exclude the possibility that the answer to Problem~\ref{pbm:main}
is $\kappa=7$.

\section{Outline of the proof}

In this section, we lay out our strategy for proving Theorem \ref{thm:main}.  Let us start off by noting that Theorem \ref{thm:main} will follow as an immediate consequence to the following theorem. 

\begin{theorem}\label{thm:9colors}
Let $G$ be a planar graph. Let $\alpha$ be a $10$-coloring of $G$. Then there exists a sequence of recolorings from $\alpha$ to some $9$-coloring of $G$ that
recolors every vertex either at most once, to a color distinct from $10$, or exactly twice, first to the color $10$ and then to a color distinct from $10$.
\end{theorem}

Theorem \ref{thm:main} follows by a standard argument.

\begin{proof}[Proof of Theorem \ref{thm:main}]
Let $\alpha$ and $\beta$ be $10$-colorings of $G$. To prove the theorem, it suffices to show that we can recolor $\alpha$ to $\beta$ by at most $8n$ recolorings. 

By Theorem~\ref{thm:9colors}, we can recolor $\alpha$ to some $9$-coloring $\alpha_1$ of $G$ by at most $2n$ recolorings and $\beta$ to some $9$-coloring $\beta_1$ by at most $2n$ recolorings.
By~\cite{thomassen}, there exists a partition of $V(G)$ into an independent set $I$ and a $3$-degenerate graph $D$.
From $\alpha_1$ and $\beta_1$ recolor the vertices in $I$ to color $10$ (the color that is not used in $\alpha_1$ and $\beta_1$).
Let $\alpha_2$ and $\beta_2$ denote the restrictions of $\alpha_1$ and $\beta_1$ to $D$.
Applying Theorem~\ref{thm:bousquet}, the distance between $\alpha_2$ and $\beta_2$ in $R_9(D)$ is at most $4|V(D)|$,
and thus we can recolor $\alpha_2$ to $\beta_2$ by at most $4|V(D)|$ recolorings without using the color $10$. This completes the proof. 
\end{proof} 

The rest of this paper will be devoted to the proof of Theorem \ref{thm:9colors}. In order to prove the theorem, we must first make a few definitions.
A \emph{scene} is a pair $(G,\alpha)$, where $G$ is a plane graph and $\alpha$ is a $10$-coloring of $G$. We say that a sequence of recolorings from $\alpha$ to some coloring $\gamma$ of $G$ is
\emph{valid} if $\gamma$ uses only colors $1, \dots, 9$ and every vertex $v$ of $G$ is recolored either at most once (to the color $\gamma(v)$) or exactly twice, first to the color $10$ and then to 
the color $\gamma(v)$. We say that the scene $(G, \alpha)$ is \emph{recolorable} if $G$ admits a valid sequence of recolorings starting from $\alpha$.  

 The scene $(G, \alpha)$  is said to be a \emph{minimal counterexample} if $(G, \alpha)$ is not recolorable and all scenes $(G', \beta)$ such that
 \begin{itemize}
 \item $|V(G')| < |V(G)|$, or
 \item $|V(G')| = |V(G)|$ and $|E(G') > |E(G)|$, or
 \item $G' = G$ and $|\beta^{-1}(10)| > |\alpha^{-1}(10)|$
 \end{itemize}
are recolorable. 
 
Our aim will be to exclude the existence of a minimal counterexample, which will prove Theorem \ref{thm:9colors}. We begin with an easy proposition.

\begin{lemma}\label{prop:1}
If $(G, \alpha)$ is a minimal counterexample, then $G$ is a triangulation and the color $10$ appears in the closed neighbourhood of every vertex of~$G$ under $\alpha$. 
\end{lemma}

\begin{proof}
Suppose that $G$ is not a triangulation; then for some face $f$ of $G$, there exist distinct non-adjacent vertices $u$ and $v$ incident with $f$. If $\alpha(u) \not = \alpha(v)$, we insert the edge $uv$. If $\alpha(u) = \alpha(v)$ we identify $u$ and $v$ into a new vertex $u'$.  
The resulting graph $G'$ is planar and, by minimality, $(G', \alpha)$ is recolorable (we consider $\alpha$ to be a coloring of $G'$ by defining $\alpha(u') = \alpha(u) = \alpha(v)$). As any valid sequence of recolorings in $G'$ easily translates into a valid sequence of recolorings in $G$, this shows that $G$ must be a triangulation. 
 
Suppose that the color $10$ does not appear on some vertex $v$ of $G$ or any of its neighbors. Recolor $v$ to the color $10$ and let $\alpha'$ denote the resulting coloring. By minimality, $(G, \alpha')$ is recolorable. It follows, by definition, that $(G, \alpha)$ is recolorable. 
\end{proof}

We now analyze the structure of a minimal counterexample $(G,\alpha)$ by showing that $G$ cannot contain a number of induced subgraphs
whose vertices are of prescribed degrees (here and in Section \ref{sec:counter}). 
Afterwards, we will show that no such minimal counterexample exists using the discharging method (see Section \ref{sec:discharging}).    

Let $H$ be an induced subgraph of $G$.  By the minimality of $(G,\alpha)$, there exists a valid sequence of recolorings in $G-V(H)$ from the restriction of $\alpha$ to $G-V(H)$
to some coloring $\gamma$ of $G-V(H)$. Let us define a list assignment $L^H$ for $H$ by setting
$$
L^H(v) = \{1, \dots, 9\} \setminus \bigg(\bigcup_{u \in N_G(v)\setminus V(H)} \{\alpha(u), \gamma(u)\} \bigg)
$$ 
for each $v \in V(H)$.  We say that $L^H$ is an \emph{assignment of available colors} to $H$ in $(G,\alpha)$; let us remark that there may be several different
assignments of available colors, corresponding to different colorings of $G-V(H)$.

We have the following proposition. A sequence of recolorings of $H$ is said be a \emph{once-only recoloring} if every vertex of $H$ is recolored at most once.
The induced subgraph $H$ of $G$ is said to be \emph{reducible} in $(G,\alpha)$ if for every assignment of available colors $L^H$ to $H$,
there exists a once-only recoloring of $H$ from the restriction of $\alpha$ to some $L^H$-coloring of $H$. 

\begin{lemma}\label{prop:2}
In a minimal counterexample $(G, \alpha)$, no induced subgraph of $G$ is reducible. 
\end{lemma}
\begin{proof}
Let $H$ be an induced subgraph of $G$. By minimality, $G-V(H)$ has a valid sequence of recolorings $\sigma$ to some coloring $\gamma$.
Let $L^H$ be the corresponding assignment of available colors to $H$.
Suppose for a contradiction $H$ is reducible. Then there exists a once-only recoloring $\sigma'$ of $H$ from the restriction of $\alpha$ to some $L^H$-coloring $\gamma_H$ of $H$.
But $\sigma'$ followed by $\sigma$ is a valid sequence of recolorings in $G$.  Indeed, recoloring of a vertex $v\in V(H)$ according to $\sigma'$ does not conflict with the colors of its neighbors
in $G-V(H)$, since if $u\in V(G)\setminus V(H)$ and $uv\in E(G)$, then $\alpha(u)\not\in L^H(v)$.  Afterwards, recolorings of $u\in V(G)\setminus V(H)$ do not conflict with the color
of its neighbors $v\in V(H)$, since $u$ can only be recolored to $10$ or $\gamma(u)$ and neither of these colors belongs to $L^H(v)$.  This is a contradiction.
\end{proof}

It is often convenient to focus just on the sizes of the lists.  For a function $s:X\to\mathbb{N}$ with $V(H)\subseteq X$, we say that a list assignment $L$ for $H$ is an \emph{$s$-list assignment} if
$|L(v)|\ge s(v)$ for every $v\in V(H)$.
Let
$$s^H_G(v)=9 - 2(\deg_G v - \deg_H v)$$ and
$$s^H_{G,\alpha}(v)=9 - 2(\deg_G v - \deg_H v)+|(N_G(v)\cap\alpha^{-1}(10))\setminus V(H)|$$
for $v\in V(H)$.

\begin{remark}\label{rem:size}
Notice, by definition, that any assignment of available colors to $H$ in $(G,\alpha)$ is an $s^H_{G,\alpha}$-list assignment, and thus also an $s^H_G$-list assignment.
\end{remark}

A \emph{motif} $M$ consists of a graph $H_M$, a $10$-coloring $\alpha_M$ of $H_M$, and an assignment $L_M$ of subsets of $\{1,\ldots,9\}$ to vertices of $H_M$.
For an induced subgraph $F$ of $H$, a motif $M'$ is an \emph{$F$-restriction} of $M$ if $H_{M'}=F$, $\alpha_{M'}$ is the restriction of $\alpha_M$ to $F$,
and $L_{M'}(v)\subseteq L_{M}(v)$ for $v\in V(F)$.
The motif $M$ is \emph{oo-recolorable (to $\gamma$)} if there exists a once-only recoloring of $H_M$ from $\alpha_M$ to an $L_M$-coloring $\gamma$ of $H_M$.
For a scene $(G,\alpha)$ and an induced subgraph $H$ of $G$, we say a motif $M$ is \emph{induced by $H$} if
$H_M=H$ and $\alpha_M$ is the restriction of $\alpha$ to $H$, and $L_H$ is an $s^H_{G,\alpha}$-list assignment.
We use the following easy consequence of Lemma~\ref{prop:2} and Remark~\ref{rem:size} to constrain minimal counterexamples.

\begin{lemma}\label{lemma:redu}
Let $(G, \alpha)$ be a minimal counterexample.  If $H$ is an induced subgraph $H$ of $G$, then there exist a motif $M$ induced by $H$ in $(G,\alpha)$
that is not oo-recolorable.
\end{lemma}
\begin{proof}
Let $\alpha_H$ be the restriction of $\alpha$ to $H$.
By Lemma~\ref{prop:2}, $H$ is not reducible, and thus for some assignment $L^H$ of available colors to $H$ in $(G,\alpha)$,
there does not exist any once-only recoloring from $\alpha_H$ to an $L^H$-coloring of $H$.
Let $M$ be the motif with $H_M=H$, $\alpha_M=\alpha_H$, and $L_M=L^H$.  Then $M$ is not oo-recolorable,
and since $L^H$ is an $s^H_{G,\alpha}$-list assignment by Remark~\ref{rem:size}, the motif $M$ is induced by $H$.
\end{proof}

In the next section, we show a number of motifs that are oo-recolorable, and thus they cannot be induced in a minimal counterexample.
Before we do that, let us point out the aspects of our argument that we consider to be novel:  Our original plan was to restrict
ourselves to once-only recolorings; this enables us to apply the method of reducible configurations which has not been previously
used in the area, since we only need to forbid two colors (the initial and the final color) per neighbor outside of the configuration.
A bit of a breakthrough for us then was the seemingly counterintuitive notion of valid sequences of recolorings,
where we introduce new vertices of color 10 in order to eventually eliminate the color 10.
This enables us to assume that color $10$ appears in the closed neighborhood of every vertex,
which is extremely useful in proving the reducibility of configurations.

\section{Structure of minimal counterexample}\label{sec:counter}

In this section, we show in a series of lemmas that if $(G,\alpha)$ is a minimal counterexample,
then $G$ has minimum degree at least five and does not contain any of the graphs in Figure \ref{fig:reducible} as induced subgraphs with the prescribed degrees of vertices.
Let us start with a trivial observation.

\begin{observation}\label{obs-start}
Suppose $M$ is a motif.  If $|V(H_M)|=1$ and $|L_M(v)|\ge 1$ for the unique vertex $v\in V(H_M)$, then $M$ is oo-recolorable.
\end{observation}

\begin{corollary}\label{cor-deg}
If $(G, \alpha)$ is a minimal counterexample, then $G$ has minimum degree at least five.
\end{corollary}
\begin{proof}
Consider a vertex $v\in V(G)$.  By Lemma~\ref{lemma:redu}, there exists a motif $M$ induced by $v$ that is not oo-colorable, and thus $|L_M(v)|=0$ by Observation~\ref{obs-start}.
But $|L_M(v)|\ge s^v_G(v)=9-2\deg_G v$, implying $\deg_G v\ge 5$.
\end{proof}
 
In order to facilitate the proofs that the graphs in Figure \ref{fig:reducible} are reducible, we first require a number of auxiliary lemmas.
Consider a motif $M$.  For brevity, let $V(M)=V(H_M)$, and for $v\in V(M)$, let $N_M(v)=N_{H_M}(v)$ and $\deg_M v=\deg_{H_M} v$.
Let us also define $\deg'_M(v)=\deg_M v-|\alpha^{-1}(10)\cap N_M(v)|$ as the number of neighbors of $v$ in $M$ whose color is not $10$.
For a vertex $v\in V(M)$, let $M-v$ denote the $(H_M-v)$ restriction of $M$ with $L_{M-v}$ equal to the restriction of $L_M$ to $H_M-v$.

\begin{lemma}\label{lem:2d}
Let $M$ be a motif and let $v$ be a vertex of $M$.
If $|L_M(v)| > \deg_M v +\deg'_M v$ and $M-v$ is oo-recolorable, then $M$ is oo-recolorable.  
\end{lemma} 
\begin{proof}
By assumptions, $M-v$ is oo-recolorable to some coloring $\gamma$, via a sequence $\sigma$ of recolorings.
Since $|L_M(v)| > \deg_M v +\deg'_M v$ and $10\not\in L_M(v)$, there exists a color $c\in L_M(v)\setminus \bigcup_{u \in N_M(v)}\{\alpha(u), \gamma(u)\}$.
Hence, we can first recolor $v$ to $c$ and then perform the recolorings according to $\sigma$, showing that $M$ is oo-recolorable.
\end{proof}

Similarly, we obtain the following observation.

\begin{lemma}\label{lem:10}
Let $M$ be a motif and let $v$ be a vertex of $M$.
If $\alpha_M(v) = 10$ and $|L_M(v)| > \deg_M v$ and $M-v$ is oo-recolorable, then $M$ is oo-recolorable.  
\end{lemma}

\begin{proof}
By assumptions, $M-v$ is oo-recolorable to some coloring $\gamma$, via a sequence $\sigma$ of recolorings.
We can first perform the recolorings $\sigma$ in $M$, as they do not conflict with the color $10$ of $v$.  Finally, we can
recolor $v$ to a color in $L_M(v)\setminus \bigcup_{u \in N_M(v)}\{\gamma(u)\}$, which exists since $|L(v)| >\deg_M v$.
This shows $M$ is oo-recolorable.
\end{proof}

For a motif $M$, a vertex $v\in V(H_M)$, and a color $c$, let $M-(v\to c)$ denote the $(H_M-v)$-restriction of $M$
with $L_{M-(v\to c)}(u)$ equal to $L_M(u)\setminus c$ for $u\in N_M(v)$ and to $L_M(u)$ for all other vertices.
\begin{lemma}\label{lem:xybas}
Let $M$ be a motif, let $v$ be a vertex of $M$, and consider any color $c\in L_M(v)\setminus\bigcup_{u\in N_M(v)} \{\alpha(u)\}$.
If the motif $M-(v\to c)$ is oo-recolorable, then $M$ is oo-recolorable.  
\end{lemma}
\begin{proof}
By assumptions, $M-(v\to c)$ is oo-recolorable via a sequence $\sigma$ of recolorings.
We can first recolor $v$ to $c$ (since no neighbor of $v$ has color $c$) and then perform the recolorings $\sigma$ in $C$.
For a neighbor $u$ of $v$, the recoloring of $u$ according to $\sigma$ does not conflict with the color $c$, since $c\not\in L_{M-(v\to c)}(u)$.
This shows $M$ is oo-recolorable.
\end{proof}

Lemma~\ref{lem:xybas} has the following useful consequence.  For a motif $M$ and a vertex $v\in V(M)$, let
$j_M:V(M)\to\mathbb{N}$ denote the function such that $j_M(u)=1$ if $u\in N_M(v)$ and $|L_M(u)|\ge 2$ and $j_M(u)=0$ otherwise.
Let $s_M:V(M)\to\mathbb{N}$ be defined by $s_M(v)=|L_M(v)|$ for $v\in V(M)$.

\begin{lemma}\label{lem:xycons}
Let $M$ be a motif and let $v$ be a vertex of $M$ such that $|L_M(v)| > \deg'_M v+|\{u \in N_M(v): |L_M(u)|=1\}|$.
If $M$ is not oo-recolorable, then there exists an $(H_M-v)$-restriction $M'$ of $M$
such that $L_{M'}$ is an $(s_M-j_M(v))$-list assignment and $M'$ is not oo-recolorable.
\end{lemma}
\begin{proof}
By assumptions, there exists a color $c\in L_M(v)\setminus\Bigl(\bigcup_{u\in N_M(v)} \{\alpha(u)\}\cup \bigcup_{u\in N_M(v), |L_M(u)|\le 1} L_M(u)\Bigr)$,
and by Lemma~\ref{lem:xybas}, we can set $M'=M-(v\to c)$.
\end{proof}

In particular, repeatedly applying Lemma~\ref{lem:xycons} until a motif with single vertex is obtained and using Observation~\ref{obs-start},
we have the following consequence.
\begin{corollary}\label{cor-triv}
Let $M$ be a motif.  If $|L_M(v)| > \deg_M v$ for every $v\in V(M)$, then $M$ is oo-recolorable.
\end{corollary}

For a motif $M$, a vertex $v\in V(M)$, and a color $c\in L_M(v)$, let $M\ominus (v\to c)$ denote the $(H_M-v)$-restriction of $M$
with $L_{M\ominus(v\to c)}(u)$ equal to $L_M(u)\setminus (\{\alpha_M(v),c\})$ for $u\in N_M(v)$ and to $L_M(u)$ for all other vertices.
In case that $|L_M(v)|=1$, we write $M\ominus v$ for brevity, since the color $c$ is uniquely determined in this case.

\begin{lemma}\label{lem:res}
Let $M$ be a motif, let $v$ be a vertex of $M$, and consider any color $c\in L_M(v)$.
If the motif $M\ominus (v\to c)$ is oo-recolorable, then $M$ is oo-recolorable.  
\end{lemma}
\begin{proof}
By assumptions, $M\ominus (v\to c)$ is oo-recolorable via a sequence $\sigma$ of recolorings.
This sequence of recolorings can also be performed in $M$, since no neighbor of $v$ can be assigned the color $\alpha_M(v)$.
Finally, we can recolor $v$ to $c$, since no neighbor may end up with the color $c$.
This shows $M$ is oo-recolorable.
\end{proof}

We will generally repeatedly use the preceding claims to simplify the motif obtained by Lemma~\ref{lemma:redu}, often to
one contradicting Corollary~\ref{cor-triv}.  For brevity, let us introduce a notation for this kind of arguments.
Suppose vertices of a motif $M$ are labelled $v_i$ for $i \in I\subseteq \{1,\ldots, m\}$.
A vector $(s_1,\ldots, s_m)$ \emph{describes} $M$ if $s_i$ is an integer smaller or equal to $|L(v_i)|$ for $i\in I$ and $s_i=\bullet$ for $i\in \{1,\ldots,m\}\setminus I$.
Furthermore, a segment of this vector can be enclosed in square brackets; this indicates that there exists an index $i$ in this segment such that $\alpha_M(v_i)=10$.
By $M\sim (s_1,\ldots, \underline{s_i}, \ldots, s_m)\xrightarrow{\text{L$n$}}(s'_1,\ldots, s'_m)\sim M'$, we mean the following:
The motif $M$ is described by the vector $(s_1, \ldots, s_m)$, and applying Lemma~$n$ with $v=v_i$, we obtain a motif $M'$ described by $(s'_1,\ldots, s'_m)$,
such that if $M$ is not oo-colorable, then $M'$ also is not oo-colorable.
In case Lemma~\ref{lem:xycons} or Lemma~\ref{lem:res} with more than one color choice is applied, we also specify the color $c$ over the arrow.  In case the resulting motif $M'$ is not further discussed
(e.g., a contradiction with Corollary~\ref{cor-triv} is obtained), the $\sim M'$ part is omitted.  We can also chain several such statements in the natural way.
In all the arguments, we without loss of generality assume that $|L(v_i)|=s_i$, implicitly removing extra colors from the lists if needed.

Recall that by Lemma~\ref{prop:1}, the color $10$ appears in the closed neighbourhood of every vertex of a minimal counterexample. 

\begin{lemma}\label{lem:55}
Let $(G,\alpha)$ be a minimal counterexample and let $v_1$ and $v_2$ be adjacent vertices of $G$. If $\deg v_1 = \deg v_2 = 5$, then either $\alpha(v_1) = 10$ or $\alpha(v_2) = 10$. 
\end{lemma}
\begin{proof}
By Lemma~\ref{lemma:redu}, there exist a motif $M$ induced by $H=G[\{v_1,v_2\}]$ in $(G,\alpha)$ that is not oo-recolorable.
If neither $u$ nor $v$ has color $10$, then since the color $10$ appears in the closed neighbourhood of every vertex, we have $s^H_{G,\alpha}(u)\ge 2$
and $s^H_{G,\alpha}(v)\ge 2$.  However, this contradicts Corollary~\ref{cor-triv}.
\end{proof}

We also need the following three easy observations. 

\begin{lemma}\label{lem:edge}
Let $M$ be a motif such that $H_M$ is an edge with vertices $v_1$ and $v_2$. If $M$ is described by $(2, 1)$, then $M$ is oo-recolorable unless $\alpha_M^{-1}(10)\cap V(M)=\emptyset$,
$L_M(v_1) = \{\alpha_M(v_1), \alpha_M(v_2)\}$ and $L_M(v_2) = \{\alpha_M(v_1)\}$.
\end{lemma}
\begin{proof}
Suppose that $M$ is not oo-recolorable. If there exists a color $c_2 \in L_M(v_2) \setminus \{\alpha_M(v_1)\}$, then $M \sim (2, \underline{1}) \rby[c_2]{xybas} (1, \bullet)$,
contradicting Observation~\ref{obs-start}. It follows that $L_M(v_2) = \{\alpha_M(v_1)\}$. Hence, if there exists a color $c_1 \in L_M(v_1) \setminus \{\alpha_M(v_1), \alpha_M(v_2)\}$, then
$M \sim (\underline{2}, 1) \rby[c_1]{xybas} (\bullet,1)$, contradicting Observation~\ref{obs-start}.
Therefore, we have $L_M(v_1) = \{\alpha_M(v_1), \alpha_M(v_2)\}$, and in particular $\alpha_M^{-1}(10)\cap V(M)=\emptyset$.
\end{proof}

\begin{lemma}\label{lem:path}
Let $M$ be a motif such that $H_M$ is a path $v_1 v_2 v_3$. If $M$ is described by $(2, 2, 2)$ and  $\alpha^{-1}(10) \cap V(H_M) \neq \emptyset$, then $M$ is oo-recolorable. 
\end{lemma}

\begin{proof}
Suppose for a contradiction $M$ is not oo-recolorable and that $\alpha^{-1}(10) \cap V(H_M) \neq \emptyset$. If $\alpha(v_1) = 10$, then $M \sim ([\underline{2}], 2, 2) \rby{10} (\bullet, 2, 2)$, contradicting Corollary~\ref{cor-triv}.
It follows by symmetry that $\alpha(v_2) = 10$; but then $M \sim (\underline{2}, [2], \underline{2}) \rby{2d} (\bullet, [2], \bullet)$, contradicting Observation~\ref{obs-start}. 
\end{proof}

\begin{lemma}\label{lem:path1}
Let $M$ be a motif such that $H_M$ is a path $v_1v_2v_3$. If $M$ is described by $(1, 4, 1)$, then $M$ is oo-recolorable. 
\end{lemma}

\begin{proof}
Suppose $M$ is not oo-recolorable. If there exists a color $c_1 \in L_M(v_1) \setminus \{\alpha_M(v_2)\}$, then $M \sim (\underline{1}, 4, 1) \rby[c_1]{xybas} (\bullet, \underline{3}, 1) \rby{2d} (\bullet, \bullet, 1)$, contradicting Observation \ref{obs-start}. So we can assume by symmetry that $L_M(v_1) = L_M(v_3) = \{\alpha_M(v_2)\}$.
But then for $c_2 \in L_M(v_2) \setminus \{\alpha_M(v_1), \alpha_M(v_2), \alpha_M(v_3)\}$, we have
$M \sim (1, \underline{4}, 1) \rby[c_2]{xybas} (1,\bullet, 1)$, contradicting Corollary~\ref{cor-triv}.
\end{proof}

We now make two observations about triangles in a minimal counterexample.

\begin{lemma}\label{lem:triangle}
Let $(G,\alpha)$ be a minimal counterexample. If $G$ contains a triangle $T$ with vertices $v_1$, $v_2$, and $v_3$ such that $v_1$ has degree five and
$v_2$ and $v_3$ have degree at most six, then $\alpha^{-1}(10)\cap V(T)\neq\emptyset$.
\end{lemma}
\begin{proof}
By Lemma~\ref{lemma:redu}, there exists a motif $M$ induced by $T$ in $(G,\alpha)$ that is not oo-recolorable.
Suppose for a contradiction no vertex of $T$ has color $10$. Since the color $10$ appears in the closed neighbourhood of every vertex, we have $s^T_{G,\alpha}(v_1)\ge 4$
and $s^T_{G,\alpha}(v_2),s^T_{G,\alpha}(v_3)\ge 2$.  If there existed a color $c\in L_M(v_2)\setminus \{\alpha(v_1),\alpha(v_3)\}$,
we would have $M\sim (4,\underline{2},2)\rby[c]{xybas}(\underline{3},\bullet,1)\rby{2d}(\bullet,\bullet,1)$, contradicting Observation~\ref{obs-start}.
Therefore, $L_M(v_2)=\{\alpha(v_1),\alpha(v_3)\}$, and by symmetry, $L_M(v_3)=\{\alpha(v_1),\alpha(v_2)\}$.  Then, letting $c'$ be a color in $L_M(v_1)\setminus\{\alpha(v_1),\alpha(v_2),\alpha(v_3)\}$,
we have $M\sim (\underline{4},2,2)\rby[c']{xybas}(\bullet,2,2)$, contradicting Corollary~\ref{cor-triv}.
\end{proof}

\begin{lemma}\label{lem:triangle1}
Let $M$ be a motif such that $H_M$ is a triangle with vertices $v_1$, $v_2$, and $v_3$.
If $M$ is described by $(4, 3, 1)$, then $M$ is oo-recolorable, and if
$M$ is described by $(3,3,1)$ or $(3,3,2)$, then $M$ is oo-recolorable unless 
$\alpha_M^{-1}(10)=\emptyset$ and $L_M(v_1) = L_M(v_2) =  \{\alpha_M(v_1), \alpha_M(v_2), \alpha_M(v_3)\}$ and $L_M(v_3) \subseteq \{\alpha_M(v_1), \alpha_M(v_2)\}$.
\end{lemma}

\begin{proof}
Suppose first $M$ is described by $(3,3,1)$ or $(3,3,2)$, and that $M$ is not oo-recolorable.  If there exists $c_3\in L_M(v_3)\setminus \{\alpha_M(v_1), \alpha_M(v_2)\}$,
then $M\sim(3,3,\underline{1})\rby[c_3]{xybas}(2,2,\bullet)$, contradicting Corollary~\ref{cor-triv}.
Hence, we have $L_M(v_3)\subseteq\{\alpha_M(v_1), \alpha_M(v_2)\}$, and by symmetry we can assume $\alpha_M(v_1)\in L_M(v_3)$.
If there exists a color $c_1\in L_M(v_1)\setminus \{\alpha_M(v_1), \alpha_M(v_2),\alpha_M(v_3)\}$,
then we can first recolor $v_1$ by $c_1$, then $v_3$ by $\alpha_M(v_1)$ and finally $v_2$ by a color in $L_M(v_2)\setminus \{\alpha_M(v_1),c_1\}$,
showing that $M$ is oo-recolorable, a contradiction.  Therefore $L_M(v_1) = \{\alpha_M(v_1), \alpha_M(v_2), \alpha_M(v_3)\}$, and in particular $\alpha_M^{-1}(10)=\emptyset$.
If $L_M(v_2)\neq L_M(v_1)$, then there would exist $c_2\in L_M(v_2)\setminus \{\alpha_M(v_1), \alpha_M(v_2), \alpha_M(v_3)\}$, and
$M\sim (3,\underline{3},1)\rby[c_2]{xybas}(\underline{3},\bullet,1)\rby{2d}(\bullet,\bullet,1)$, contradicting Observation~\ref{obs-start}.
This gives the characterization of non-oo-recolorable motifs described by $(3,3,1)$ or $(3,3,2)$.

Suppose now $M$ is described by $(4,3,1)$; then we can delete a color from $L_M(v_1)$ to obtain a motif $M'$ described by $(3,3,1)$, but with $L_{M'}(v_1)\neq L_{M'}(v_2)$.
The motif $M'$ is oo-recolorable by the previous paragraph, and thus $M$ is oo-recolorable as well.
\end{proof}

We also require the following observation on diamonds in a minimal counterexample.

\begin{lemma}\label{lem:fan}
Let $(G, \alpha)$ be a minimal counterexample.  Let $v_1$, \ldots, $v_4$ be distinct vertices of $G$
such that the subgraph $F$ of $G$ induced by $\{v_1,v_2,v_3,v_4\}$ contains all possible edges except for $v_2v_4$.
If $\deg v_1\le 7$, $\deg v_2\le 5$ and $\deg v_3,\deg v_4\le 6$, then $\alpha^{-1}(10) \cap V(F) \neq 10$. 
\end{lemma}

\begin{proof}
By Lemma~\ref{lemma:redu}, there exists a motif $M$ induced by $F$ in $(G,\alpha)$ that is not oo-recolorable.
Suppose for a contradiction no vertex of $F$ has color $10$. Since the color $10$ appears in the closed neighbourhood of every vertex, $M$ is described by $(2, 4, 4, 2)$. If there exists a color $c_4 \in L_M(v_4) \setminus \{\alpha(v_1), \alpha(v_3)\}$, then $M \sim (2, 4, 4, \underline{2}) \rby[c_4]{xybas} (1, 4, 3, \bullet)$, contradicting Lemma~\ref{lem:triangle1}.
Therefore $L_M(v_4) = \{\alpha(v_1), \alpha(v_3)\}$.  If there exists a color $c_1\in L_M(v_1)\setminus \{\alpha(v_2),\alpha(v_3),\alpha(v_4\}$, then
$M \sim (\underline{2}, 4, 4, 2) \rby[c_1]{xybas} (\bullet, \underline{3}, 3, 1)\rby{2d}(\bullet, \bullet, \underline{3}, 1)\rby{2d}(\bullet, \bullet, \bullet, 1)$, contradicting Observation~\ref{obs-start}.
Hence, $L_M(v_1)\subseteq \{\alpha(v_2),\alpha(v_3),\alpha(v_4\}$.
If there exists a color $c_3 \in L_M(v_3) \setminus \{\alpha(v_1),\ldots,\alpha(v_4)\}$, then $M \sim (2, 4, \underline{4}, 2) \rby[c_3]{xybas} (2, \underline{3}, \bullet, 2) \rby{2d} (2, \bullet, \bullet,  2)$,
contradicting Corollary~\ref{cor-triv}.  Therefore, $L_M(v_3) = \{\alpha(v_1),\ldots,\alpha(v_4)\}$, and in particular $\alpha(v_2)\neq\alpha(v_4)$.
Choose a color $c_2 \in L_M(v_2) \setminus \{\alpha(v_1),\alpha(v_2),\alpha(v_3)\}$.
\begin{itemize}
\item If $\alpha(v_2)\in L_M(v_1)$, we first recolor $v_2$ to $c_2$, then $v_1$ to $\alpha(v_2)$, and finally $v_4$ to $\alpha(v_1)$.
\item Otherwise, $L_M(v_1)=\{\alpha(v_3),\alpha(v_4)\}$.  We first recolor $v_2$ to $c_2$, then $v_3$ to $\alpha(v_2)$, then $v_1$ to $\alpha(v_3)$, and finally $v_4$ to $\alpha(v_1)$.
\end{itemize}
\end{proof}

\begin{figure}
\begin{center}
\scalebox{1.2}{\includegraphics[scale=0.8]{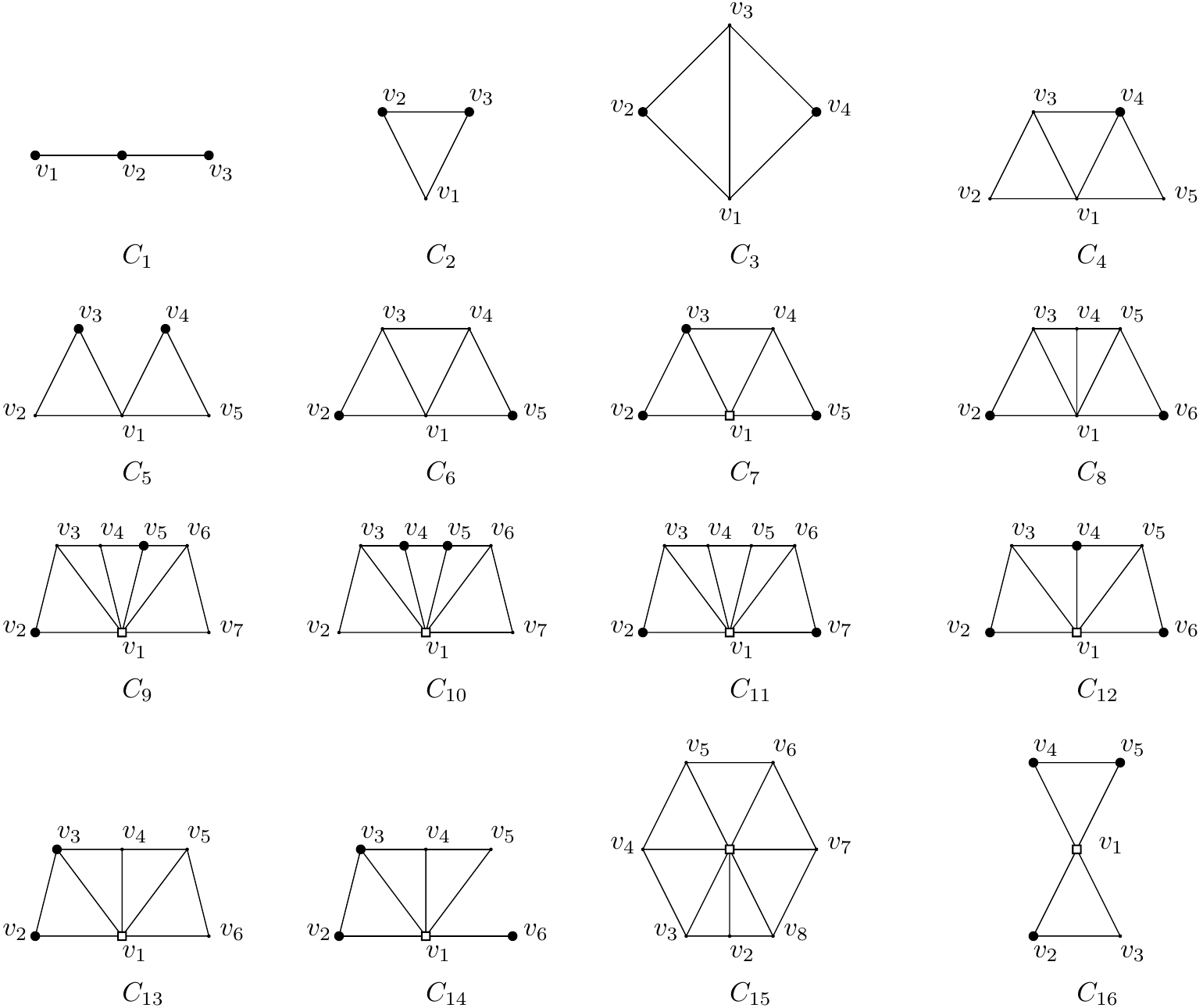}}
\end{center}
\caption{Reducible induced subgraphs, where $\square$ denotes a vertex of degree at most seven, $\bullet$ denotes a vertex of degree five and $\cdot$ denotes a vertex of degree at most six.}
\label{fig:reducible}
\end{figure}

We are now ready to demonstrate that the graphs in Figure \ref{fig:reducible} are reducible. 

\begin{lemma}\label{lem:reducible}
If $(G,\alpha)$ is a minimal counterexample, then $G$ contains none of the induced subgraphs with prescribed vertex degrees depicted in 
Figure~\ref{fig:reducible}.
\end{lemma}

\begin{proof}
Suppose for a contradiction $C$ is one of the graphs depicted in Figure~\ref{fig:reducible} and contained in $G$ as an induced subgraph
with the prescribed degrees of vertices.  
By Lemma~\ref{lemma:redu}, there exist a motif $M$ induced by $C$ in $(G,\alpha)$ that is not oo-recolorable.
We prove that each of the cases are reducible separately, starting with $C_1$ and working our way towards $C_{16}$. 
We fix the labelling of vertices as indicated in Figure~\ref{fig:reducible}.
\begin{enumerate}[{\normalfont \hspace{-33pt} (C1)}]
\item By Lemma~\ref{lem:55}, either $\alpha(v_1)=\alpha(v_3)=10$, or $\alpha(v_2)=10$.  In the former case,
$M\sim([1],\underline{3},[1])\rby{2d}(1,\bullet,1)$, contradicting Corollary~\ref{cor-triv}.
In the latter case, $M\sim(1,[\underline{3}],1)\rby{10}(1,\bullet,1)$, again contradicting Corollary~\ref{cor-triv}.

\item By Lemma~\ref{lem:55} and symmetry, we can assume $\alpha(v_3)=10$.
But then $M\sim (1,3,[\underline{3}])\rby{10}(1,\underline{3},\bullet)\rby{2d}(1,\bullet,\bullet)$, contradicting Observation~\ref{obs-start}.

\item If $\alpha(v_2)=\alpha(v_4)=10$, then $M\sim(3,[\underline{3}],3,[\underline{3}])\rby{10}(3,\bullet,3,\bullet)$, contradicting Corollary~\ref{cor-triv}.
Hence, by Lemma~\ref{lem:triangle} and symmetry, we can assume $\alpha(v_3)=10$.
However, then $M\sim(\underline{3},3,[3],3)\rby{xycons}(\bullet,\underline{2},[2],\underline{2})$, contradicting Lemma~\ref{lem:path}.

\item If $\alpha(v_1)=10$, then $M\sim([\underline{5}],1,3,5,1)\rby{10}(\bullet,1,3,\underline{5},1)\rby{2d}(\bullet, 1,\underline{3}, \bullet,1)\rby{2d}(\bullet, 1,\bullet, \bullet,1)$, contradicting Corollary~\ref{cor-triv}.
If $\alpha(v_3)=\alpha(v_5)=10$, then $M\sim(5,1,[3],\underline{5},[1])\rby{2d}(\underline{5},1,[3],\bullet,[1])\rby{2d}(\bullet, 1,[\underline{3}], \bullet,1)\rby{10}(\bullet, 1,\bullet, \bullet,1)$, contradicting Corollary~\ref{cor-triv}.
Hence, Lemma~\ref{lem:triangle} implies $\alpha(v_4)=10$, and thus $M\sim(5,1,3,[\underline{5}],1)\rby{10}(5,1,3,\bullet,1)\sim M'$.  Let $\{c_5\}=L_{M'}(v_5)$.
We have $M'\sim(5,1,3,\bullet,\underline{1})\rby{res}(3,1,3,\bullet,\bullet)$, and thus Lemma~\ref{lem:triangle1} implies that $L_M(v_1)$ is the disjoint union of $L_M(v_3)=\{\alpha(v_1),\alpha(v_2),\alpha(v_3)\}$
and $\{\alpha(v_5),c_5\}$.  In particular, $c_5\neq\alpha(v_1)$, and thus $M'\sim(5,1,3,\bullet,\underline{1})\rby[c_5]{xybas}(4,1,3,\bullet,\bullet)$, contradicting Lemma~\ref{lem:triangle1}

\item If $\alpha(v_1) = 10$, then $M \sim (\underline{[5]}, 1, 3, 3, 1) \rby{10}(\bullet, 1, \underline{3}, \underline{3}, 1) \rby{2d}(\bullet, 1, \bullet, \bullet, 1)$, contradicting Corollary~\ref{cor-triv}.
If $\alpha(v_3) = \alpha(v_4) = 10$, then $M \sim (5, 1, [\underline{3}], [\underline{3}], 1) \rby{10}(\underline{5}, 1, \bullet, \bullet, 1) \rby{2d}(\bullet, 1, \bullet, \bullet, 1)$, again contradicting Corollary~\ref{cor-triv}.
If $\alpha(v_3) = \alpha(v_5) = 10$,  then $M \sim (5, 1, [\underline{3}], 3, [1]) \rby{10} (5, \underline{1}, \bullet, 3, [1]) \rby{res}(3, \bullet, \bullet, 3, [1])$,
contradicting Lemma~\ref{lem:triangle1}.  Hence, by Lemma~\ref{lem:triangle} and symmetry, we can assume $\alpha(v_2)= \alpha(v_5) = 10$.
But then $M \sim (5, [1], \underline{3}, \underline{3}, [1])\rby{xycons}(\underline{3}, [1], \bullet, \bullet, [1])\rby{2d}(\bullet, [1], \bullet, \bullet, [1])$,
contradicting Corollary~\ref{cor-triv}.

\item If $\alpha(v_1) = 10$, then $M \sim ([\underline{5}], 3, 3, 3, 3) \rby{10} (\bullet, 3, 3, 3, 3)$, contradicting Corollary~\ref{cor-triv}.
Note that at most one of the adjacent vertices $v_3$ and $v_4$ can have color 10.  Hence, by Lemma~\ref{lem:triangle} and symmetry, we can assume $\alpha(v_2)=10$.
But then $M \sim (5, [\underline{3}], 3, 3, 3]) \rby{10} (5, \bullet, \underline{3}, 3, 3)\rby{xycons}(4,\bullet,\bullet,2,3)$, contradicting Lemma~\ref{lem:triangle1}.

\item If $\alpha(v_2) = 10$, then $M \sim (3, [\underline{3}], 5, 3, 3) \rby{10} (3, \bullet, \underline{5}, 3, 3) \rby{2d} (3, \bullet, \bullet, 3, 3)$, which contradicts Corollary~\ref{cor-triv}.
Hence, by Lemma~\ref{lem:55} we have $\alpha(v_3) = 10$. It follows that $M \sim (3, 3, [\underline{5}], 3, 3) \rby{10} (3, \underline{3}, \bullet, 3, 3) \rby{2d} (3, \bullet, \bullet, 3, 3)$, which again contradicts Corollary~\ref{cor-triv}. 

\item If $\alpha(v_1) = 10$, then $M \sim ([\underline{7}], 3, 3, 3, 3, 3) \rby{10} (\bullet, 3, 3, 3, 3, 3)$, which contradicts~Corollary \ref{cor-triv}.
If $\alpha(v_2) = \alpha(v_6) = 10$, then $M \sim (7, [\underline{3}], 3, 3, 3, [\underline{3}]) \rby{10} (\underline{7}, \bullet, 3, 3, 3, \bullet) \rby{2d} (\bullet, \bullet, 3, 3, 3, \bullet)$, again contradicting Corollary~\ref{cor-triv}.
If $\alpha(v_3) = \alpha(v_5) = 10$, choose $c_1 \in L_m(v_1) \setminus \{\alpha(v_2), \alpha(v_4), \alpha(v_6)\}$;
we have $M \sim (\underline{7}, 3, [3], 3, [3], 3)\rby[c_1]{xybas}(\bullet, \underline{2}, [2], 2, [2], \underline{2}) \rby{2d} (\bullet, \bullet, [2], 2,  [2], \bullet)$, contradicting Lemma~\ref{lem:path}.

Hence, by Lemma~\ref{lem:triangle} and symmetry, we can assume $\alpha(v_2) = \alpha(v_5) = 10$. For $c_3 \in L_M(v_3) \setminus \{\alpha(v_1), \alpha(v_4)\}$,
we have $M \sim (7, [\underline{3}], 3, 3, [3], 3) \rby{10} (7, \bullet, \underline{3}, 3, [3], 3) \rby[c_3]{xybas} (\underline{6}, \bullet, \bullet, 2, [3], 3) \rby{2d} (\bullet, \bullet, \bullet, 2, [3], 3)$,  contradicting Corollary~\ref{cor-triv}.

\item If $\alpha(v_1) = 10$, then we have $M \sim  ([\underline{7}], 3, 3, 3, 5, 3, 1) \rby{10} (\bullet, 3, 3, 3, \underline{5}, 3, 1) \rby{2d} (\bullet, 3, 3, 3, \bullet, \underline{3}, 1) \rby{2d} (\bullet, 3, 3, 3, \bullet, \bullet, 1)$, which contradicts Corollary \ref{cor-triv}. Therefore by Lemma \ref{lem:fan} we can assume that at least one of $v_2$, $v_3$, $v_4$ has color $10$ and at least one of $v_5$, $v_6$, $v_7$ has color $10$.
Choose a color $c_6 \in L_M(v_6) \setminus \{\alpha(v_6), \alpha(v_7)\}$, let $L_M(v_7) = \{c_7\}$, and choose a color $c_1 \in L_M(v_1) \setminus (\{c_6, c_7\} \cup \bigcup_{i = 2}^7\{\alpha(v_i)\})$.
Then $M \sim (\underline{7}, [3, 3, 3], [5, 3, 1]) \rby[c_1]{xybas} (\bullet, [2, 2, 2], [4, 2, 1]) \sim M'$, where $L_{M'}(v_6)\neq\{\alpha(v_6),\alpha(v_7)$.

If $\alpha(v_5) = 10$, then we can continue with $M' \sim (\bullet, [2, 2, 2], [\underline{4}], 2, 1) \rby{10} (\bullet, [2, 2, 2], \bullet, \underline{2, 1}) \rby{edge} (\bullet, [2, 2, 2], \bullet, \bullet, \bullet)$, which contradicts Lemma~\ref{lem:path}.
If $\alpha(v_6) = 10$, then $M' \sim (\bullet, [2, 2, 2], 4, [2], \underline{1}) \rby{xycons} (\bullet, [2, 2, 2], 4, [\underline{1}], \bullet) \rby{xycons} (\bullet, [2, 2, 2], \underline{3}, \bullet, \bullet) \rby{2d} (\bullet, [2, 2, 2], \bullet, \bullet, \bullet)$, which again contradicts Lemma \ref{lem:path}.
Finally, suppose $\alpha(v_7) = 10$. Then  $M' \sim (\bullet, [2, 2, 2], 4, 2, [\underline{1}]) \rby{xycons} (\bullet, [2, 2, 2], 4, 1, \bullet) \sim M^*$.

If $\alpha(v_4) = 10$, then $M^* \sim (\bullet, 2, 2, [2], \underline{4}, 1, \bullet) \rby{2d} (\bullet, \underline{2, 2, [2]}, \bullet, 1, \bullet) \rby{path} (\bullet, \bullet, \bullet, \bullet, \bullet, 1, \bullet)$, which contradicts Observation~\ref{obs-start}.
If $\alpha(v_3) = 10$, then $M^* \sim (\bullet, \underline{2}, [2], 2, 4, 1, \bullet) \rby{2d} (\bullet, \bullet, [\underline{2}], 2, 4, 1) \rby{10} (\bullet, \bullet, \bullet, 2, 4, 1)$, which contradicts Lemma~\ref{lem:path1}. 
If $\alpha(v_2) = 10$, then $M^* \sim (\bullet, [\underline{2}], 2, 2, 4, 1, \bullet) \rby{10} (\bullet, \bullet, \underline{2}, 2, 4, 1) \rby{xycons} (\bullet, \bullet, \bullet, 1, 4, 1)$, which again contradicts Lemma~\ref{lem:path1}. 

\item  By Lemma~\ref{lem:55} and symmetry, we can assume $v_5$ has color $10$.  By Lemma~\ref{lem:fan}, it follows that $v_2$ or $v_3$ has color $10$.
We have $M \sim (7, [1, 3], 5, [\underline{5}], 3, 1) \rby{10} (7, [1, 3], \underline{5}, \bullet, 3, 1) \rby{2d} (7, [1, 3], \bullet, \bullet, 3, 1) \sim M'$.
Let $L_{M'}(v_i) = \{c_i\}$ for $i\in\{2,7\}$ and choose $c_6 \in L_{M'}(v_6) \setminus \{\alpha(v_6), \alpha(v_7)\}$.
Then there exists a color $c_1 \in L_{M'}(v_1) \setminus \{c_2, c_6, c_7, \alpha(v_2), \alpha(v_3), \alpha(v_6), \alpha(v_7)\}$, and
$M' \sim (\underline{7}, [1, 3], \bullet, \bullet, 3, 1) \rby[c_1]{xybas} (\bullet, [1, 2], \bullet, \bullet, 2, 1)\sim M^*$, where $L_{M^*}(v_6)\neq\{\alpha(v_6),\alpha(v_7)\}$.
This contradicts Lemma~\ref{lem:edge}.

\item If $\alpha(v_1)=10$, then $M\sim([\underline{7}], 3, 3, 3, 3, 3, 3)\rby{10}(\bullet, 3, 3, 3, 3, 3, 3)$, contradicting Corollary~\ref{cor-triv}.
Lemma~\ref{lem:fan} thus implies $\alpha^{-1}(10)\cap\{v_2,v_3,v_4\}\neq\emptyset$ and $\alpha^{-1}(10)\cap\{v_5,v_6,v_7\}\neq\emptyset$.
Apply Lemma~\ref{lem:10} to the vertices in $\alpha^{-1}(10)\cap\{v_2,v_7\}$ and Lemma~\ref{lem:xycons}
to the vertices $v_i$ such that $i\in\{2,\ldots,7\}$, $\alpha(v_i)\neq 10$ and $\deg'_M(v_i)\le 2$; let $M'$ denote the resulting motif.  
Suppose that for some $i\in \{2,3,4\}$, we have $v_i\in V(M')$ and $\alpha(v_i)\neq 10$; the construction of $M'$ implies $i\neq 2$ and $\alpha(v_{i-1})\neq 10\neq \alpha(v_{i+1})$,
and since $\alpha^{-1}(10)\cap\{v_2,v_3,v_4\}\neq\emptyset$, it follows that $i=4$ and $\alpha(v_2)=10$.  By a symmetric argument for $\{5,6,7\}$, we
conclude that $\deg'_{M'}(v_1)\le |\alpha^{-1}(10)\cap \{v_2,v_7\}|$.
However, since $|L_M(v_1)|>\deg_M(v_1)$, the construction of $M'$ implies $|L_{M'}(v_1)|>\deg_{M'} v_1+|\alpha^{-1}(10)\cap \{v_2,v_7\}|$.
Therefore, $M'$ is oo-recolorable by Lemma~\ref{lem:2d} applied to $v_1$ and by Corollary~\ref{cor-triv}.
This is a contradiction.

\item If $\alpha(v_1) = 10$, let $c_1 \in L_M(v_1)$. We have $M \sim ([\underline{5}], 3, 3, 5, 3, 3) \rby[c_1]{res} (\bullet, \underline{2}, 2, 4, 2, \underline{2}) \rby{xycons} (\bullet, \bullet, 1, 4, 1, \bullet)$, contradicting Lemma \ref{lem:path1}.
If $\alpha(v_4) = 10$, we have $M \sim (5, 3, 3, [\underline{5}], 3, 3) \rby{10} (\underline{5}, 3, 3, \bullet, 3, 3) \rby{xycons} (\bullet, 2, 2, \bullet, 2, 2)$, which contradicts Corollary~\ref{cor-triv}.
If $\alpha(v_3) = 10$, then $M \sim (\underline{5}, 3, 3, 5, [3], 3)\rby{xycons} (\bullet, 2,2,\underline{4},[2],2)\rby{2d}(\bullet, 2,2,\bullet,[2],2)$, contradicting Corollary~\ref{cor-triv}.
The case $\alpha(v_5)=10$ is symmetric.  Therefore, Lemma~\ref{lem:fan} implies $\alpha(v_2)=\alpha(v_6)=10$, and thus
$M \sim (5, [\underline{3}], 3, 5, 3, [\underline{3}]) \rby{10} (5,\bullet, 3,5,3,\bullet)$, contradicting Corollary~\ref{cor-triv}.

\item By Lemma~\ref{lem:55}, either $\alpha(v_2) = 10$ or $\alpha(v_3) = 10$, and thus either
$M\sim (5,[\underline{3}],5,3,3,1)\rby{10}\sim (5,\bullet,\underline{5},3,3,1)\rby{2d} (5,\bullet,\bullet,3,3,1)\sim M'$,
or $M\sim (5,3,[\underline{5}],3,3,1)\rby{10}\sim (5,\underline{3},\bullet,3,3,1)\rby{2d} (5,\bullet,\bullet,3,3,1)\sim M'$.
Let $\{c_6\}=L_{M'}(v_6)$; we have $M'\sim (5,\bullet,\bullet,3,3,\underline{1})\rby[c_6]{res}(3,\bullet,\bullet,3,1,\bullet)\sim M^*$,
and by Lemma~\ref{lem:triangle1}, we have $L_{M^*}(v_1)=L_{M'}(v_4)=\{\alpha(v_1),\alpha(v_4),\alpha(v_5)\}$.
Consequently, $L_{M'}(v_1)=\{\alpha(v_1),\alpha(v_4),\alpha(v_5),\alpha(v_6),c_6\}$, and in particular $c_6\not\in\{\alpha(v_1),\alpha(v_5)\}$.
Therefore $M'\sim (5,\bullet,\bullet,3,3,\underline{1})\rby[c_6]{xybas}(4,\bullet,\bullet,3,2,\bullet)$, contradicting Lemma~\ref{lem:triangle1}.

\item By Lemma~\ref{lem:55}, either $\alpha(v_2) = 10$ or $\alpha(v_3) = 10$, and thus either
$M\sim (5,[\underline{3}],5,3,1,1)\rby{10}\sim (5,\bullet,\underline{5},3,1,1)\rby{2d} (5,\bullet,\bullet,3,1,1)\sim M'$,
or $M\sim (5,3,[\underline{5}],3,1,1)\rby{10}\sim (5,\underline{3},\bullet,3,1,1)\rby{2d} (5,\bullet,\bullet,3,1,1)\sim M'$.
Let $\{c_6\}=L_{M'}(v_6)$; we have $M'\sim (5,\bullet,\bullet,3,1,\underline{1})\rby[c_6]{res}(3,\bullet,\bullet,3,1,\bullet)\sim M^*$,
and by Lemma~\ref{lem:triangle1}, we have $L_{M^*}(v_1)=L_{M'}(v_4)=\{\alpha(v_1),\alpha(v_4),\alpha(v_5)\}$.
Consequently, $L_{M'}(v_1)=\{\alpha(v_1),\alpha(v_4),\alpha(v_5),\alpha(v_6),c_6\}$, and in particular $c_6\neq \alpha(v_1)$.
Therefore $M'\sim (5,\bullet,\bullet,3,1,\underline{1})\rby[c_6]{xybas}(4,\bullet,\bullet,3,1,\bullet)$, contradicting Lemma~\ref{lem:triangle1}.

\item In this case $M$ is described by $(9, 3, 3, 3, 3, 3, 3, 3)$.
Repeatedly apply Lemma~\ref{lem:xybas} to the vertices $v_2, \dots, v_8$ as long as there exists $i\in\{2,\ldots,8\}$ such that the list of $v_i$ contains a color not appearing on its neighbors;
let $M'$ denote the resulting motif.  Note that $|L_{M'}(v_1)|>|V(M')|$ and that $|L_{M'}(v_i)|=\deg_{M'} v_i$ and $L_{M'}(v_i)\subseteq \alpha(V(M'))$ for $i\in\{2,\ldots,8\}$ such that $v_i\in V(M')$.
Hence, there exists a color $c_1\in |L_{M'}(v_1)|\setminus\alpha(V(M'))$, and this color does not appear in the lists of vertices of $\{v_2,\ldots,v_8\}\cap V(M')$.
Applying Lemma~\ref{lem:xybas}, $M'-(v_1\to c_1)$ contradicts Corollary~\ref{cor-triv}.
 
\item By Lemma~\ref{lem:55} and symmetry, we can assume that $\alpha(v_2) = 10$.
If $\alpha^{-1}(10)\cap \{v_4,v_5\}\neq \emptyset$,
then $M \sim (3, [\underline{3}], 3, [3, 1]) \rby{10} (3, \bullet, \underline{3}, [3,1]) \rby{2d} (3, \bullet, \bullet, [3,1])$, which contradicts Lemma~\ref{lem:triangle1}.
Therefore, the color $10$ does not appear in the closed neighborhood of $v_4$ in $C$.  Since the color $10$ appears in the closed neighborhood of every vertex in $G$,
we have $s^C_{G, \alpha}(v_4) \geq 4$, and thus
$M \sim (3, [\underline{3}], 3, 4, 1) \rby{10} (3, \bullet, \underline{3}, 4,1) \rby{2d} (3, \bullet, \bullet, 4,1)$, which contradicts Lemma~\ref{lem:triangle1}.
\end{enumerate}
\end{proof}

\section{Discharging phase}\label{sec:discharging}

Consider a plane triangulation $G$, a vertex $v\in V(G)$ of degree $k\ge 3$, and its neighbors $v_1$, \ldots, $v_k$
in the clockwise order around $G$.  We say that the subgraph of $G$ consisting of the cycle $v_1\ldots v_k$, the vertex $v$,
and the edges $vv_i$ for $i=1,\ldots, k$ is a \emph{wheel}, $v$ is its \emph{center} and $v_1, \dots, v_k$ its \emph{rim}.  Note that a wheel is not necessarily an induced subgraph of $G$.
Let $T$ be the triangle bounding the outer face of $G$.
Let $C$ be a graph and $d:V(C)\to\mathbb{N}$ a function assigning a \emph{prescribed degree} to each vertex of $C$.
We say that $C$ with the prescribed degrees $d$ \emph{appears} in $G$ if there exists a wheel $W$ in $G$ and
an injective function $f:V(C)\to V(W)$ such that
\begin{itemize}
\item for distinct $x,y\in V(C)$, $xy$ is an edge of $C$ if and only if $f(x)f(y)$ is an edge of $W$,
\item for all $x\in V(C)$,
 $\deg_G f(x)  \leq d(x)$, and
\item $f(V(C))\cap V(T)=\emptyset$.
\end{itemize}
Hence, $C$ is an induced subgraph of $W$, but not necessarily of $G$ 
 (since $W$ may not be an induced subgraph of $G$).   Let us remark that the last technical condition
from the definition of appearance will be later used to deal with this issue.

\begin{lemma}\label{lemma-appear}
Suppose $G$ is a plane triangulation such that every vertex not incident with the outer face of $G$ has degree at least five.
If $|V(G)|\ge 4$, then one of the graphs with prescribed degrees depicted in Figure~\ref{fig:reducible}
appears in $G$.
\end{lemma}
\begin{proof}
Suppose for a contradiction none of these graphs appears in $G$.
We assign the initial charge $\mathrm{ch}_0(v) = 10 \cdot \deg v - 60$ to each vertex $v$ of $G$.   Since $G$ is
a triangulation, we have $|E(G)|=3|V(G)|-6$ by Euler's formula, and thus
\begin{equation}\label{eq:1}
\sum_{v \in V(G)} \mathrm{ch}_0(v) = 20|E(G)| - 60|V(G)| =  - 120.
\end{equation}

A vertex is \emph{big} if it either has degree at least 7 or it is incident with the outer face of $G$,
\emph{medium} if it has degree six and is not incident with the outer face of $G$,
and \emph{small} if it has degree five and is not incident with the outer face of $G$.
Next, we redistribute the charges according to the following rules.
For accounting purposes, for a rule sending some amount of charge from a vertex $v$ to another vertex $u$,
we also specify faces incident with $v$ through which the charge leaves $v$, and
an edge $e$ incident with $u$ along which the charge arrives to $u$.  Additionally, we specify a face incident with $e$ through which
the charge passes.

\begin{itemize}
\item[\normalfont{(R1)}] A big vertex $v$ sends $2$ units of charge to each adjacent small vertex $u$ along the edge $vu$;
of this charge, one unit leaves $v$ and passes through one of the faces incident with the edge $uv$, while the other unit leaving $v$ passes through the other face incident with $uv$. 
\item[\normalfont{(R2)}] Suppose $vux$ is a face of $G$, $v$ is big, $u$ is small and $x$ is medium or small.
Then $v$ sends $1$ unit of charge to $u$; the charge leaves $v$ and passes through the face $vux$ to arrive to $u$ along the edge $xu$.
\item[\normalfont{(R3)}] Suppose $v_1$, \ldots, $v_m$ for some $m\in\{3,\ldots,6\}$ are consecutive neighbors of a medium vertex $x$
in the clockwise or the counterclockwise order, $v_1$ is small, $v_2$, \ldots, $v_{m-1}$ are medium and $v_m$ is big.
Then $v_m$ sends $1$ unit of charge to $v_1$; the charge leaves $v_m$ through the face $xv_{m-1}v_m$ and passes through the face
$xv_1v_2$ to arrive to $v_1$ along the edge $v_2v_1$.
\end{itemize}
Note that (R2) applies in addition to the two units of charge sent by $v$ to $u$ by (R1),
but the charge arrives to $u$ along a different edge.  Furthermore, if $x$ is small, the charge is also being
sent from $v$ to $x$ by (R2) with the roles of $u$ and $x$ exchanged.
Furthermore, note that (R3) may possibly send charge from $v_m$ to $v_1$ twice around the same vertex $x$, once in the clockwise direction, once in the counterclockwise one
(when $x$ is the center of a wheel whose rim contains $v_1$ and $v_m$ and every other vertex of the rim is medium).
We now analyze the final charge $\mathrm{ch}(v)$ of each vertex $v$ of $G$ after the redistribution of the charge.
Clearly, for a medium vertex $v$, we have $\mathrm{ch}(v)=\mathrm{ch}_0(v)=0$.

Consider now a small vertex $z$.  We claim that for each edge $e=wz$ incident with $z$ and each face $f=wzx$ incident with $e$,
a unit of charge passes through $f$ to arrive to $z$ along $e$, and thus $\mathrm{ch}(z)=\mathrm{ch}_0(z)+10\times 1=0$.
Indeed, if $w$ is big, then this is the case by (R1).  If $w$ is not big and $x$ is big,
then a unit of charge passing through $f$ arrives to $z$ along $e$ from $x$ by (R2).  If neither $w$ nor $x$ is big, then since $C_2$
does not appear in $G$, both of them are medium.  Since $C_4$ does not appear in $G$, $x$ has a neighbor distinct from $z$ that is not
medium.  Let $v_1=z$, $v_2=w$, $v_3$, \ldots, $v_m$ be the neighbors of $x$ in order, where $v_3$, \ldots, $v_{m-1}$ are medium and $v_m$
is not medium.  Since $C_3$, $C_6$, $C_8$ and $C_2$ do not appear in $G$, the vertex $v_m$ is not small, and thus $v_m$ is big.
Consequently, a unit of charge passing through $f$ arrives to $z$ along $e$ from $v_m$ by (R3).

Suppose now $v$ is a vertex of degree $d\ge 7$ not incident with the outer face of $G$.
For a face $f=vxy$, let $t(f)$ denote the total amount of charge that leaves $v$ through $f$.
If both $x$ and $y$ are small, then $t(f)=4$ since two units leave through $f$ by (R1), one along the edge $vx$ and the other along $vy$, and two by (R2), both along the edge $xy$.  If $x$ is small and $y$ is medium or vice versa,
then $t(f)=2$ since one unit leaves through $f$ by (R1) and one by (R2).  If both $x$ and $y$ are medium, then $t(f)\le 2$, since
at most two units leave through $f$ by (R3).  If $x$ is small and $y$ is big or vice versa, then $t(f)=1$, since only
one unit leaves through $f$ by (R1).  Otherwise, $t(f)=0$.

Furthermore, consider the faces $f_1$ and $f_2$ following $f$ in the clockwise order around $f$.  Since $C_1$ does not appear in $G$,
if $t(f)=4$, then $t(f_1)\le 2$ and $t(f_2)\le 2$.  Consequently, there are at most $\lfloor d/3\rfloor$ faces $f$ incident with $v$
such that $t(f)=4$.  If $d\ge 8$, this implies
$$\mathrm{ch}(v)\ge \mathrm{ch}_0(v)-2d-2\lfloor d/3\rfloor=8d-2\lfloor d/3\rfloor-60\ge 0.$$
Hence, we can assume $d=7$, and thus $\mathrm{ch}_0(v)=10$.  Let $v_1$, \ldots, $v_7$ be the neighbors of $v$ in the clockwise order, and for $i=1,\ldots,7$,
let $f_i$ be the face $vv_iv_{i+1}$ (where $v_8=v_1$).  Let $s=\sum_{i=1}^7 t(f_i)$ be the total amount of charge sent by $v$.
We argue that $s\le 10$, and thus $\mathrm{ch}(v)=\mathrm{ch}_0(v)-s\ge 0$.  To do so, we discuss several cases.
\begin{itemize}
\item \emph{$v$ is adjacent to two consecutive small vertices in the cycle on neighbors of $v$.} Thus $v$ is incident with a face $f$ such that $t(f) = 4$. By symmetry, we can assume $t(f_1)=4$, and
thus $v_1$ and $v_2$ are small.  Since $C_1$ does not appear in $G$, $v_3$ and $v_7$ are not small.  

If $v_5$ is small,
then since $C_{16}$ does not appear in $G$, both $v_4$ and $v_6$ are big and hence $t(f_4)=t(f_5)=1$, $t(f_3)=t(f_6)=0$, and $t(f_2),t(f_7)\le 2$,
implying $s\le 10$.  Hence, we can assume $v_5$ is not small.

Suppose $v_6$ and $v_7$ are both medium.  Since $C_{13}$ does not appear in $G$, $v_5$ is big, and thus $t(f_7)+t(f_6)+t(f_5)\le 2+2+0=4$.
Since $C_{14}$ and $C_{10}$ do not appear in $G$, $v_4$ is not small and $v_3$ and $v_4$ are not both medium, respectively, implying $t(f_3)=0$ and $t(f_4)=0$.
Consequently, $s\le 4+2+0+0+4=10$.  Hence, assume $v_6$ and $v_7$ are not both medium, and symmetrically, that $v_3$ and $v_4$
are not both medium.

If $v_4$ is small, then since $C_7$ and $C_{16}$ do not appear in $G$, $v_3$ and $v_5$ are big and
$t(f_2)+t(f_3)+t(f_4)=1+1+1=3$.  Otherwise, since $v_3$ and $v_4$ are not both medium, we have $t(f_3)=0$ and $t(f_2)+t(f_4)\le 3$.
Hence $t(f_2)+t(f_3)+t(f_4)\le 3$, and symmetrically $t(f_7)+t(f_6)+t(f_5)\le 3$.  It follows that $s\le 4+3+3=10$.

\item \emph{small vertices are not consecutive in the cycle on neighbors of $v$.} Consequently, $t(f)\le 2$ for each face incident with $v$ and
$v$ is adjacent to at most three small vertices.

Before we proceed, let us make a useful observation: 

\noindent($\star$) \emph{For any $b\in \{1,\ldots, 5\}$, 
if none of $v_b$, $v_{b+1}$ and $v_{b+2}$ is small, then $t(f_b)+t(f_{b+1})\le 3$.}

This is clearly the case unless $v_b$, $v_{b+1}$, and $v_{b+2}$ are all medium and $t(f_b)=t(f_{b+1})=2$.
Then, let $v_b$, $v$, $v_{b+2}$, $z_3$, $z_2$, $z_1$ be the neighbors of $v_{b+1}$ in order.
Since $t(f_b)=t(f_{b+1})=2$, charge leaves $v$ through $f_b$ and $f_{b+1}$ twice by (R3),
and thus either both $z_1$ and $z_3$ are small, or none of $z_1$, $z_2$, and $z_3$ is big and at least one of them is small.
But then either $C_5$ or $C_4$ appears in $G$, which is a contradiction.

Let us now continue with the case analysis.
\begin{itemize}
\item \emph{$v$ is adjacent to three small vertices.} By symmetry we can assume $v_1$, $v_3$, and $v_5$ are small.
Since $C_{12}$ does
not appear in $G$, we can by symmetry assume $v_2$ is big hence $t(f_1) = t(f_2) = 1$.  If $v_4$ is big, then $t(f_3) = t(f_4) = 1$ implying $s\le 4\times 1+3\times 2=10$. Thus, since $C_1$ does
not appear in $G$, we can assume $v_4$ is medium.  Since $C_9$ does not appear in $G$, $v_6$ and $v_7$ cannot both be medium, and thus
$t(f_6)=0$.  Consequently, $s\le 1+1+2+2+2+0+2=10$.
\item \emph{$v$ is adjacent to two small vertices, at distance two in the cycle on neighbors of $v$.}  By symmetry we can assume $v_1$ and $v_3$
are small.  If $v_5$ is big, then $t(f_4)=t(f_5)=0$ and $s\le 5\times 2=10$.  Hence, we can assume $v_5$ is medium, and by symmetry $v_6$ is medium.
Since $C_{11}$ does not appear in $G$, $v_4$ and $v_7$ are not both medium; by symmetry, we can assume $v_7$ is big,
and thus $t(f_6)=0$ and $t(f_7)=1$.  Furthermore, $t(f_4)+t(f_5)\le 3$ by ($\star$), and thus $s\le 2+2+2+3+0+1=10$.
\item \emph{$v$ is adjacent to two small vertices, at distance three in the cycle on neighbors of $v$.} By symmetry we can assume $v_1$ and $v_4$
are small.  If $v_6$ is big or both $v_5$ and $v_7$ are big, then $t(f_5)=t(f_6)=0$ and $s\le 5\times 2=10$;
hence, we can by symmetry assume $v_5$ and $v_6$ are medium.  Since $C_9$ does not appear in $G$, $v_2$ and $v_3$ are not both medium, and thus
$t(f_1)+t(f_2)+t(f_3)\le 2+0+1=3$.  Furthermore, $t(f_5)+t(f_6)\le 3$ by ($\star$), implying $s\le 3+2+3+2=10$.
\item \emph{$v$ is adjacent to at most one small vertex.}  By symmetry we can assume no neighbor of $v$ other than $v_1$ is small.
If $v_i$ is big for some $i\in\{1,3,4,5,6\}$, then $t(f_{i-1})=t(f_i)=0$ (where $f_0=f_7$) and $s\le 5\times 2=10$.
Hence, we can assume $v_i$ is medium for $i\in\{3,4,5,6\}$ and $v_1$ is medium or small.  Since $C_{15}$ does not appear in $G$,
$v_2$ and $v_7$ are not both medium; by symmetry, we can assume $v_2$ is big, and thus $t(f_1)+t(f_2)\le 1$.  By ($\star$),
we have $t(f_3)+t(f_4)\le 3$, and thus $s\le 1+3+3\times 2=10$.
\end{itemize}
\end{itemize}
We conclude that every vertex not incident with the outer face of $G$ has non-negative final charge.

Finally, let us consider a vertex $v$ incident with the outer face of $G$.  Since $|V(G)|\ge 4$ and $G$ is a triangulation, we have
$\deg(v)\ge 3$.  Furthermore, the outer face $f$ of $G$ is incident only with big vertices by definition, and thus $t(f)=0$.
In the utmost case, $t(f') \leq 4$ for every face $f' \not= f$ incident with $v$ and hence $\mathrm{ch}(v)\ge \mathrm{ch}_0(v)-(\deg v-1)\times 4=6\deg v-56\ge-38$.
Therefore, (\ref{eq:1}) together with the fact that no charge is created or lost in the redistribution process gives
$$-120=\sum_{v \in V(G)} \mathrm{ch}_0(v)=\sum_{v \in V(G)} \mathrm{ch}(v)\ge 3\times(-38),$$
which is a contradiction.
\end{proof}

\begin{corollary}\label{cor-indu}
If $G$ is a plane triangulation of minimum degree at least five, then
one of the graphs depicted in Figure~\ref{fig:reducible} is an induced
subgraph of $G$ with prescribed vertex degrees.
\end{corollary}
\begin{proof}
If $G$ contains a separating triangle, then let $T$ be a separating triangle in $G$
such that the open disk in the plane bounded by $T$ is minimal; otherwise, let $T$ be
the triangle bounding the outer face of $G$.  Let $G'$ be the induced subgraph of $G$ drawn in the
closed disk bounded by $T$.  By Lemma~\ref{lemma-appear}, one of the graphs $C$ with prescribed
degrees depicted in Figure~\ref{fig:reducible} appears in $G$, via a map $f:V(C)\to V(W)$ for a
wheel $W$ in $G'$.  By the choice of $G'$, observe that $G'$ does not contain any separating triangle,
and thus $W$ is an induced subgraph of $G'$, and thus also of $G$.  Since $C$ is an induced subgraph
of $W$, it follows that $C$ is an induced subgraph of $G$.  Furthermore, $V(C)\cap V(T)=\emptyset$
by the last condition from the definition of appearance, and thus the vertices of $f(V(C))$ have the
same degree in $G'$ and in $G$.
\end{proof}

The proof of the main result is now straightforward.

\begin{proof}[Proof of Theorem~\ref{thm:9colors}]
Suppose for a contradiction that there exists a non-recolorable scene $(G,\alpha)$.
Choose such a scene with the smallest number of vertices, among those with the largest number of edges, and among those
with the largest number of vertices of color $10$.  Then $(G,\alpha)$ is a minimal counterexample,
and thus $G$ is a triangulation by Lemma~\ref{prop:1}, has minimum degree at least five by Corollary~\ref{cor-deg},
and does not contain any of the induced subgraphs with prescribed vertex degrees depicted in
Figure~\ref{fig:reducible}.  However, this contradicts Corollary~\ref{cor-indu}.
\end{proof}

\section*{Acknowledgements} 

Zden\v{e}k Dvo\v{r}\'ak was supported in part by ERC Synergy grant DYNASNET no. 810115.
Carl Feghali was supported by grant 19-21082S of the Czech Science Foundation.

\bibliography{bibliography}{}
\bibliographystyle{abbrv}
 
\end{document}